\documentclass{article}
\usepackage[utf8]{inputenc}
\pdfoutput=1
\usepackage[sc,osf]{mathpazo}
\usepackage[margin=1in,letterpaper]{geometry}
\usepackage[sort&compress,numbers]{natbib}
\usepackage[colorlinks=true,citecolor=blue,breaklinks]{hyperref}
\usepackage[hyphenbreaks]{breakurl} 
\usepackage{nicefrac}

\makeatletter
\newlength\aftertitskip     \newlength\beforetitskip
\newlength\interauthorskip  \newlength\aftermaketitskip

\def\maketitle{\par
 \begingroup
   \def\thefootnote{\fnsymbol{footnote}}
   \def\@makefnmark{\hbox to 4pt{$^{\@thefnmark}$\hss}}
   \@maketitle \@thanks
 \endgroup
\setcounter{footnote}{0}
 \let\maketitle\relax \let\@maketitle\relax
 \gdef\@thanks{}\gdef\@author{}\gdef\@title{}\let\thanks\relax}

\def\@startauthor{\noindent \normalsize\bf}
\def\@endauthor{}
\def\@starteditor{\noindent \small {\bf Editor:~}}
\def\@endeditor{\normalsize}
\def\@maketitle{\vbox{\hsize\textwidth
 \linewidth\hsize \vskip \beforetitskip
 {\begin{center} \LARGE\@title \par \end{center}} \vskip \aftertitskip
 {\def\and{\unskip\enspace{\rm and}\enspace}%
  \def\addr{\small\it}%
  \def\email{\hfill\small\tt}%
  \def\name{\normalsize\bf}%
  \def\AND{\@endauthor\rm\hss \vskip \interauthorskip \@startauthor}
  \@startauthor \@author \@endauthor}
}}

\makeatother

\pdfoutput=1                    

\usepackage{graphicx}
\usepackage{amsmath,amsthm,amssymb,bm} 
\usepackage{amsfonts}
\usepackage{xspace}
\usepackage{array}
\usepackage{enumerate}
\usepackage[boxruled]{algorithm2e}
\usepackage{stmaryrd}
\usepackage{colortbl}

\numberwithin{equation}{section}

\newcommand{\reals}{\mathbb{R}}

\newcommand{\E}{\mathbb{E}}

\newtheorem{theorem}{Theorem}
\newtheorem{corr}[theorem]{Corollary}
\newtheorem{prop}[theorem]{Proposition}
\newtheorem{lemma}[theorem]{Lemma}

\theoremstyle{definition}

\theoremstyle{remark}

\definecolor{dkgreen}{rgb}{0,0.6,0}
\definecolor{gray}{rgb}{0.5,0.5,0.5}
\definecolor{mauve}{rgb}{0.58,0,0.82}
\definecolor{cdarkred}{RGB}{180,0,0} 
\definecolor{cdarkgreen}{RGB}{0,120,0} 
\definecolor{mywhite}{RGB}{255,255,255} 
\definecolor{myyellow}{RGB}{238,238,  0} 
\definecolor{mygreen}{RGB}{124,252,  0} 
\definecolor{myblue}{RGB}{10,0,255} 
\definecolor{mypurple}{RGB}{159,  0,238} 
\definecolor{myred}{RGB}{238, 64,  0} 
\definecolor{cdarkblue}{RGB}{0,0,102}
\definecolor{sharpred}{RGB}{255,0,0}
\newcommand{\darkred}[1]{{\color{cdarkred} #1}}

\newcommand{\nlsum}{\sum\nolimits}
\newcommand{\nlprod}{\prod\nolimits}
\newcommand{\half}{\tfrac12}

\newcommand{\pfrac}[2]{\left(\tfrac{#1}{#2}\right)}

\newcommand{\vx}{\bm{x}}
\newcommand{\vy}{\bm{y}}

\begin{document}
\title{New concavity and convexity results for symmetric polynomials and their ratios}
\author{\name Suvrit Sra \email{suvrit@mit.edu}\\
  \addr{Massachusetts Institute of Technology}}
\maketitle

\begin{abstract}
  We prove some ``power'' generalizations of Marcus-Lopes-style (including McLeod and Bullen) concavity inequalities for elementary symmetric polynomials, and convexity inequalities (of McLeod and Baston) for complete homogeneous symmetric polynomials. Finally, we present sundry concavity results for elementary symmetric polynomials, of which  the main result is a concavity theorem that among other implies a well-known log-convexity result of Muir (1972/74) for positive definite matrices. 
\end{abstract}

\section{Introduction}
The main purpose of this note is to prove the following inequalities: 
\begin{subequations}
  \begin{align}
    \label{eq:one}
    [e_k((\vx+\vy)^p)]^{\nicefrac1{pk}}
    &\ge [e_k(\vx^p)]^{\nicefrac1{pk}} + [e_k(\vy^p)]^{\nicefrac1{pk}}\\
    \label{eq:three}
    [h_k((\vx+\vy)^{1/p})]^{\nicefrac{p}{k}}
    &\le
    [h_k(\vx^{1/p})]^{\nicefrac{p}{k}} + [h_k(\vy^{1/p})]^{\nicefrac{p}{k}}\\
    \label{eq:two}
    \Bigl[\frac{e_k((\vx+\vy)^p)}{e_{k-l}((\vx+\vy)^p)}\Bigr]^{\frac1{lp}}
    &\ge \Bigl[\frac{e_k(\vx^p)}{e_{k-l}(\vx^p)}\Bigr]^{\frac1{lp}} +
    \Bigl[\frac{e_k(\vy^p)}{e_{k-l}(\vy^p)}\Bigr]^{\frac1{lp}}\\
    \label{eq:last}
    \Bigl[\frac{h_k((\vx+\vy)^{1/p})}{h_{1}((\vx+\vy)^{1/p})}\Bigr]^{\frac{p}{k-1}}
    &\le \Bigl[\frac{h_k(\vx^{1/p})}{h_{1}(\vx^{1/p})}\Bigr]^{\frac{p}{k-1}} +
    \Bigl[\frac{h_k(\vy^{1/p})}{h_{1}(\vy^{1/p})}\Bigr]^{\frac{p}{k-1}},
  \end{align}
\end{subequations}
for $\vx, \vy \in \reals_+^n$, $p \in (0,1)$ and $1\le l \le k \le n$. In these inequalities $e_k$ denotes the $k$-th elementary symmetric polynomial and $h_k$ denotes the $k$-th complete homogeneous symmetric polynomial, and $\vx^p$ we mean the vector $(x_1^p,x_2^p,\ldots,x_n^p)$. 
Inequalities~\eqref{eq:one}-\eqref{eq:last} offer a ``power'' generalization to their corresponding counterparts with $p=1$ that have been well-known for a long time, with attributions to \citet{marcus1957,mcleod59,baston1978,bullen1961}, among others---see also the reference book~\citep{bullen2013}.

\vskip5pt
\noindent\darkred{\bf Outline:} We begin by proving \eqref{eq:one} and \eqref{eq:two} in Section~\ref{sec:ml}. Subsequently, in Section~\ref{sec:ekmore} we prove additional concavity results for $e_k$, notable among which is a new concavity theorem that implies the log-convexity result of \citet{muir74} as a corollary. Finally, in Section~\ref{sec:hk} we prove inequalities~\eqref{eq:three} and \eqref{eq:last}.

\section{Generalized Marcus-Lopes-style inequalities}
\label{sec:ml}
Let $\vx \in \reals_+^n$ and $e_k(\vx) = \sum_{S \subseteq [n], |S|=k}\prod_{i\in S}x_i$ denote the $k$-th \emph{elementary symmetric polynomial}; we set $e_0=1$. A well-known and important concavity result for $e_k$ is the \emph{Marcus-Lopes inequality}~\citep{marcus1957}:
\begin{equation}
  \label{eq:ml.orig}
  \frac{e_k(\vx+\vy)}{e_{k-1}(\vx+\vy)}
  \ge
  \frac{e_k(\vx)}{e_{k-1}(\vx)}
  +
  \frac{e_k(\vy)}{e_{k-1}(\vy)},\qquad 1\le k \le n,\quad\vx, \vy \in \reals_+^n.
\end{equation}
Inequality~\eqref{eq:ml.orig} is used by \citet{marcus1957} to prove the concavity of $e_k(\vx)^{1/k}$. We establish below the following concavity inequality:
\begin{equation}
  \label{eq:ml.new}
  \biggl[\frac{e_k((\vx+\vy)^p)}{e_{k-1}((\vx+\vy)^p)}\biggr]^{\nicefrac1{p}}
  \ge
  \biggl[\frac{e_k(\vx^p)}{e_{k-1}(\vx^p)}\biggr]^{\nicefrac1{p}}
  +
  \biggl[\frac{e_k(\vy^p)}{e_{k-1}(\vy^p)}\biggr]^{\nicefrac1{p}},
\end{equation}
Using~\eqref{eq:ml.new} as a key step, we will first establish \eqref{eq:one} and then extend it to prove inequality \eqref{eq:two}.

An important building block of our proofs is provided by the \emph{parallel sum} operation\footnote{The parallel sum operation is also defined for positive definite matrices, where we set $\bm{A}:\bm{B} := (\bm{A}^{-1}+\bm{B}^{-})^{-1}$.}
\begin{equation}
  \label{eq:14}
  x : y := (x^{-1}+y^{-1})^{-1},\qquad x, y > 0,
\end{equation}
Lemma~\ref{lem:base} proves a key joint concavity property of~\eqref{eq:14}, as well as its ``power'' generalization.
\begin{lemma}
  \label{lem:base}
  The parallel sum $:$ defined by~\eqref{eq:14} is jointly concave on $\reals_+^2$. Moreover, for every $p\ge -1$, the `\emph{$p$-parallel sum}' $x:_py := [x^p:y^p]^{1/p}$ is jointly concave. Also, both $:$ and $:_p$ are monotonic in both arguments.
\end{lemma}
\begin{proof}
  The Hessian equals $(p+1)\left(
\begin{array}{cc}
 -x^{p-1} y^{p+1} \left(x^p+y^p\right)^{-2-\frac{1}{p}} & x^p y^p \left(x^p+y^p\right)^{-2-\frac{1}{p}} \\
 x^p y^p \left(x^p+y^p\right)^{-2-\frac{1}{p}} & -x^{p+1} y^{p-1} \left(x^p+y^p\right)^{-2-\frac{1}{p}} \\
\end{array}
\right)$, which is clearly negative definite for $p\ge -1$. Monotonicity is clear from first derivatives.
\end{proof}

Observe that $x:y=y:x$ and $(x:y):z = x: (y:z)$, thus we extend the above notation and simply write $x_1 : x_2 : \cdots : x_n \equiv  x_1 : [ x_2 : [ \cdots : (x_{n-1}:x_n)]]$. However, the operation $:_p$ is not associative if generalized the same way. Thus, a more preferable multivariate generalization is the following:
\begin{equation}
  \label{eq:2}
  (x_1,\ldots,x_n) \mapsto [x_1^p : x_2^p : \cdots : x_n^p]^{1/p}.
\end{equation}
Later in Section~\ref{sec:add} we will study joint concavity of~\eqref{eq:2} and with it some new inequalities for $e_k$. The rest of this section is devoted to proving inequality~\eqref{eq:ml.new}.
\begin{lemma}
  \label{lem:key2}
  Let $f_1: \reals^{m_1} \to \reals_{++}$ and $f_2 : \reals^{m_2} \to \reals_{++}$ be continuous concave functions. Then, $f_1(x) :_p f_2(y)$ is jointly concave on $\reals^{m_1}\times \reals^{m_2}$.
\end{lemma}
\begin{proof}
  The proof is immediate from first principles; we include details for completeness. It suffices to establish midpoint concavity. Since $f_1$ and $f_2$ are concave, we have
  \begin{align*}
    f_1\pfrac{x_1+x_2}{2} \ge \half f_1(x_1)+\half f_1(x_2),\quad\text{and}\quad
    f_2\pfrac{y_1+y_2}{2} \ge \half f_2(y_1)+\half f_2(y_2).
  \end{align*}
  The function $:_p$ is monotonically increasing in each of its arguments and is jointly concave, therefore
  \begin{align*}
    f_1\pfrac{x_1+x_2}{2} :_p f_2\pfrac{y_1+y_2}{2}
    \ge
    \left(\tfrac{f_1(x_1)+ f_1(x_2)}{2}\right) :_p \left(\tfrac{f_2(y_1)+ f_2(y_2)}{2}\right)
    \ge \half [f_1(x_1) :_p f_2(y_1) ] + \half [f_1(x_2) :_p f_2(y_2)],
  \end{align*}
  which establishes the joint concavity of $f_1(x) :_p f_2(y)$.
\end{proof}

We are now ready to present our main result.
\begin{theorem}
  \label{thm:ekratio}
  Let $1\le k \le n$. Then, the function $$\phi_{k,n}(\vx) := \left[\frac{e_k(\vx^p)}{e_{k-1}(\vx^p)}\right]^{\nicefrac1p}$$ is concave for $\vx\in \reals_+^n$ and $p\in (0,1)$.
\end{theorem}
\begin{proof}
  We will prove the claim by induction on $k$. For $k=1$, we have
  \begin{equation*}
    \phi_{1,n}(\vx)=(x_1^p+x_2^p+\cdots+x_n^p)^{1/p},
  \end{equation*}
  which is clearly concave. As the induction hypothesis, assume that $\phi_{k-1,n}$ is concave for some $k > 1$. The key step in our proof is the following remarkable observation of~\citet{anderson1984}:
  \begin{equation}
    \label{eq:1}
    \textit{Let}\ \ \psi_{k,n}(\vx) := \frac{\binom{n}{k-1}e_k(\vx)}{\binom{n}{k}e_{k-1}(\vx)},\quad\textit{then}\quad
    \psi_{k,n}=
    \sum_{j=1}^n \frac{1}{n-k+1}x_j : \frac{1}{k-1}\psi_{k-1,n}(\vx_{[j]}),
  \end{equation}
  where $\vx_{[j]}$ denotes the vector $\vx$ with $x_j$ omitted. Using representation~\eqref{eq:1} we see that for suitable (positive) scaling factors $a_k$ and $b_k$, we can write
  \begin{equation}
    \label{eq:4}
    \begin{split}
      \phi_{k,n}(\vx) &= \Bigl(\nlsum_{j=1}^n a_k^p x_j^p : b_k^p \phi_{k-1,n}^p(\vx_{[j]}) \Bigr)^{\nicefrac1p} = \Bigl(\nlsum_{j=1}^n [a_kx_j :_p b_k\phi_{k-1,n}(\vx_{[j]})]^p\Bigr)^{\nicefrac1p}.
    \end{split}
  \end{equation}
  From induction hypothesis we know that $\phi_{k-1,n}(\cdot)$ is concave; thus, applying Lemma~\ref{lem:key2} we see that $$g_j(\vx) := (a_kx_j) :_p \bigl(b_k\phi_{k-1,n}(\vx_{[j]})\bigr)$$ is jointly concave in $x_j$ and $\vx_{[j]}$ (and thus in $\vx$). Consequently, we can further rewrite~\eqref{eq:4} as
  \begin{equation}
    \phi_{k,n}(\vx) = \bigl(\nlsum_{j=1}^ng_j(\vx)^p\bigr)^{\nicefrac1p},
  \end{equation}
  which is clearly concave as it is the (vector) composition the coordinate-wise increasing concave function $(\sum_j x_j^p)^{1/p}$ with the vector $\bigl(g_1(\vx),\ldots,g_n(\vx)\bigr)$, where each $g_j(\vx)$ is itself concave. 
\end{proof}

Following an idea of~\citet{marcus1957} (who proved concavity of $e_k(\vx)^{1/k}$), we can now prove \eqref{eq:one}, that is, the concavity of $[e_k(\vx^p)]^{1/p}$, by leveraging the ratio-concavity proved in Theorem~\ref{thm:ekratio}.
\begin{theorem}
  \label{thm:cve}
  The function $\vx \mapsto [e_k(\vx^p)]^{1/pk}$ is concave for $p \in (0,1)$ and $\vx \in \reals_+^n$.
\end{theorem}
\begin{proof}
  With Theorem~\ref{thm:ekratio} in hand the proof is simple. Indeed, we have
  \begin{align*}
    [e_k((\vx + \vy)^p)]^{1/pk} &=
    \left[\frac{e_k((\vx+\vy)^p)}{e_{k-1}((\vx+\vy)^p)}\cdot \frac{e_{k-1}((\vx+\vy)^p)}{e_{k-2}((\vx+\vy)^p)}\cdots \frac{e_1((\vx+\vy)^p)}{e_0((\vx+\vy)^p)} \right]^{1/pk}\\
    &= \left[\phi_{k,n}(\vx+\vy)\phi_{k-1,n}(\vx+\vy)\cdots\phi_{1,n}(\vx+\vy)\right]^{1/k}\\
    &\ge \left[\left(\phi_{k,n}(\vx)+\phi_{k,n}(\vy)\right)\cdots\left(\phi_{1,n}(\vx)+\phi_{1,n}(\vy)\right)\right]^{1/k}\\
    &\ge \prod_{j=1}^k[\phi_{j,n}(\vx)]^{1/k} + \prod_{j=1}^k[\phi_{j,n}(\vy)]^{1/k}\\
    &= [e_k(\vx^p)]^{1/pk} + [e_k(\vy^p)]^{1/pk},
  \end{align*}
  where the first inequality follows from Theorem~\ref{thm:ekratio}, while the second is Minkowski's inequality. 
\end{proof}

In the same vein, we easily obtain a proof to~\eqref{eq:two} which generalizes Theorem~\ref{thm:ekratio} (for $p=1$, this generalization reduces to an inequality of \citet{mcleod59}; also independently of~\citet{bullen1961}):
\begin{theorem}
  \label{thm:ml.gen}
  Let $1\le l \le k \le n$, then
  \begin{equation}
    \label{eq:6}
    \Phi_{k,l,n}(\vx) := \left[\frac{e_k(\vx^p)}{e_{k-l}(\vx^p)}\right]^{1/lp},
  \end{equation}
  is concave for $\vx \in \reals_{+}^n$ for $p \in (0,1)$.
\end{theorem}
\begin{proof}
  The argument is straightforward. First, we rewrite $\Phi_{k,l,n}(\vx)$ as
  \begin{align*}
    \Phi_{k,l,n}(\vx) &= \Bigl[\frac{e_k(\vx^p)}{e_{k-l}(\vx^p)}\Bigr]^{1/lp}
    = \Bigl[\nlprod_{j=1}^l\left(\tfrac{e_{k-j+1}(\vx^p)}{e_{k-j}(\vx^p)}\right)^{1/p} \Bigr]^{1/l} = \Bigl[\nlprod_{j=1}^l\phi_{k-j+1,n}(\vx)\Bigr]^{1/l}.
  \end{align*}
  Then, using Theorem~\ref{thm:ekratio} and Minkowski's inequality we have
  \begin{align*}
    \Phi_{k,l,n}(\vx+\vy) &=
    \left[\nlprod_{j=1}^l\phi_{k-j+1,n}(\vx+\vy)\right]^{1/l}
    \ge
    \left[\nlprod_{j=1}^l \left(\phi_{k-j+1,n}(\vx)+\phi_{k-j+1,n}(\vx)\right)\right]^{1/l}\\
    &\ge
    \left[\nlprod_{j=1}^l\phi_{k-j+1,n}(\vx)\right]^{1/l} + \left[\nlprod_{j=1}^l\phi_{k-j+1,n}(\vy) \right]^{1/l}\\
    &= \Phi_{k,l,n}(\vx) + \Phi_{k,l,n}(\vy).\qedhere
  \end{align*}
\end{proof}

\section{Additional convexity / concavity results}
\label{sec:add}
In this section we provide some additional concavity results for $e_k$ before proving the remaining two convexity inequalities \eqref{eq:three} and \eqref{eq:last} for $h_k$.
\subsection{Elementary symmetric polynomials}
\label{sec:ekmore}
We begin with a simple result about joint concavity of the multivariate $p$-parallel sum.
\begin{prop}
  \label{prop:key}
  The multivariate $p$-parallel sum map $\vx \mapsto [x_1^p : x_2^p : \cdots : x_n^p]^{1/p}$ is concave on $\reals_+^n$ for $p>0$.
\end{prop}
\begin{proof}
  The proof is by induction on $n$. The case $n=1$ is trivial; the first nontrivial base case is $n=2$, which is established by Lemma~\ref{lem:base}. Assume thus that the claim holds for $\reals_+^k$ for $k=1,2,\ldots,n-1$. Consider thus, 
  \begin{equation*}
    S_n(x_1,\ldots,x_n) = [x_1^p : x_2^p : \cdots : x_n^p]^{\frac1p}.
  \end{equation*}
  From the induction hypothesis, we have that
  \begin{align}
    \label{eq:3}
    S_{n-1}\bigl(\half(x_1+y_1),\ldots,\half(x_{n-1}+y_{n-1})\bigr)
    &\ge
    \half S_{n-1}(x_1,\ldots,x_{n-1}) + \half S_{n-1}(y_1,\ldots,y_{n-1}).
  \end{align}
  Thus, using the monotonicity of the $:_p$ operation and~\eqref{eq:3} we get
  \begin{align*}
    S_n\bigl(\half(x_1+y_1),\ldots,\half(x_n+y_n)\bigr)
    &=
    \bigl\lbrace[S_{n-1}\bigl(\half(x_1+y_1),\ldots,\half(x_{n-1}+y_{n-1})\bigr)]^p : (\half x_n+\half y_n)^p\bigr\rbrace^{1/p}\\
    &\ge
    \bigl\lbrace[\half S_{n-1}(x_1,\ldots,x_{n-1}) + \half S_{n-1}(y_1,\ldots,y_{n-1})]^p : (\half x_n + \half y_n)^p\bigr\rbrace^{1/p}.
  \end{align*}
  Now concavity of $:_p$ (Lemma~\ref{lem:base}) immediately yields
  \begin{align*}
    &\left[ \left(\tfrac{S_{n-1}(x_1,\ldots,x_{n-1})+S_{n-1}(y_1,\ldots,y_{n-1})}{2}\right)^p : \pfrac{x_n+y_n}{2}^p\right]^{1/p}\\
    &\ge
    \half [ \{S_{n-1}(x_{1},\ldots,x_{n-1})\}^p : x_n^p]^{1/p} + \half [ \{S_{n-1}(y_{1},\ldots,y_{n-1})\}^p+y_n^p]^{1/p},
  \end{align*}
  establishing concavity of $S_n$. Thus, by induction the said map is concave.
\end{proof}

The following lemma is immediate by observation; we state it for subsequent use.
\begin{lemma}
  \label{lem:one}
  Let $f: \reals^n \to \reals$ be a continuous function. Then, $f$ is \emph{reciprocally concave}, i.e., $1/f$ is concave, iff $f\left(\frac{x+y}{2}\right) \le H(f(x),f(y))=2(f(x):f(y))$.
\end{lemma}

\begin{theorem}
  \label{thm:ekcve}
  The function $\vx \mapsto e_k(\vx^p)$ is reciprocally concave for $p\in (-1,0)$ and $\vx \in \reals_+^n$.
\end{theorem}
\begin{proof}
  First observe that $(x_1\cdots x_k)^r$ is jointly concave for $r \in (0,1)$, and thus $(x_1\cdots x_k)^p$ is reciprocally concave for $p\in (-1,0)$. This allows us to conclude
  \begin{align*}
    e_k\pfrac{x_i+y_i}{2}^p &= \sum_{|I|=k}\nlprod_{i \in I} \pfrac{x_i+y_i}{2}^p
    \le
    \sum_{|I|=k}H\left(\nlprod_{i \in I} x_i^p, \nlprod_{i \in I} y_i^p \right)\\
    &\le H\left(\nlsum_{|I|=k}\nlprod_{i \in I} x_i^p, \nlsum_{|I|=k}\nlprod_{i \in I} y_i^p\right) = H(e_k(\vx^p),e_k(\vy^p)),
  \end{align*}
  where the second inequality follows from the joint concavity of the $:$ operation.
\end{proof}

An immediate corollary of Theorem~\ref{thm:ekcve} is the following reciprocal concavity result for matrices.
\begin{corr}
  \label{cor:ekmtx}
  Let $\bm{X}$ be a Hermitian positive definite matrix. Then, the function
  \begin{equation*}
    \bm{X} \mapsto \frac{1}{e_k(\lambda(\bm{X}^p))},\quad p \in (-1,0),
  \end{equation*}
  is concave (here $\lambda(\cdot)$ denotes the eigenvalue map, i.e., the vector of eigenvalues of a matrix).
\end{corr}
Corollary~\ref{cor:ekmtx} in turn yields Corollary~\ref{cor:muir} (a 1974 result of \citet{muir74}) as well as its recent generalization Corollary~\ref{cor:mariet} by~\citet{mariet2017}. Both of these results easily follow since reciprocal concavity of a positive function implies log-convexity.
\begin{corr}[\protect{\citep{muir74}}]
  \label{cor:muir}
  Let $\bm{A} \in \mathbb{C}^{n\times k}$ ($k\le n$) have full column rank, and $\bm{X}\in \mathbb{P}_n$ (the set of Hermitian positive definite matrices). The function $\bm{X} \mapsto e_k((\bm{A}^*\bm{XA})^{-1})$ is log-convex.
\end{corr}
\begin{corr}[\protect{\citep{mariet2017}}]
  \label{cor:mariet}
  Let $\bm{A} \in \mathbb{C}^{n\times k}$ ($k\le n$) have full column rank. Let $\bm{X}\in \mathbb{P}_n$ and $p\in (-1,0)$. Then, the function $\bm{X} \mapsto e_k((\bm{A}^*\bm{XA})^p)$ is log-convex.
\end{corr}

\subsection{Complete homogeneous symmetric polynomials}
\label{sec:hk}
In this section, we prove \eqref{eq:three} and \eqref{eq:last}, the power generalization to existing subadditivity inequalities for complete homogeneous symmetric polynomials. We follow a proof technique very different from \citet{mcleod59} and \citet{baston1978}, which we believe offers a much simpler proof of these inequalities.

Key to our proof is the following convenient representation for complete homogeneous symmetric polynomials (we learned of this representation from~\citep{barvinok2005}):
\begin{lemma}
  \label{lem:hkint}
  Let $\xi_1,\ldots,\xi_n$ be independent standard exponential random variables. Then, for any integer $k\ge 0$
  \begin{equation}
    \label{eq:9}
    \frac{1}{k!}\E[(\xi_1x_1+\cdots+\xi_nx_n)^k] = \sum_{1\le i_1\le i_2 \le \cdots \le i_k \le n}\nlprod_{j=1}^k x_{i_j} = h_k(\vx).
  \end{equation}
\end{lemma}
\begin{proof}
  Since $\E[\xi_i^{j}] = j!$, by expanding $(\xi_1x_1+\cdots+\xi_nx_n)^k$ and taking expectations, the result follows.
\end{proof}
Representation~\eqref{eq:9} permits an immediate proof of the following convexity result of~\citet{mcleod59} (\emph{cf.} the concavity result for $e_k(\vx)^{1/k}$, which is much harder).
\begin{equation}
  \label{eq:10}
  [h_k(\vx+\vy)]^{1/k} \le [h_k(\vx)]^{1/k} + [h_k(\vy)]^{1/k}.
\end{equation}
We prove below a ``fractional power'' generalization of~\eqref{eq:10} below; we will need the following lemma.
\begin{lemma}[Mixed-Minkowski]
  \label{lem:mink}
  Let $p\ge 1$ and $k\ge 1$. Let $\vx,\vy\in \reals_+^n$. Then,
  \begin{equation}
    \label{eq:7}
    \Bigl[\nlsum_i\left(\nlsum_j x_{ij}^p+y_{ij}^p\right)^k\Bigr]^{1/pk}
    \le
    \Bigl(\nlsum_{ij}x_{ij}^{pk}\Bigr)^{1/pk} + \Bigl(\nlsum_{ij} y_{ij}^{pk}\Bigr)^{1/pk}.
  \end{equation}
\end{lemma}
\begin{proof}
  Consider $a,b,c,d\ge 0$. Since $p \ge 1$, we have $a^p+b^p \le (a+b)^p$, whereby $(a^p+b^p)^k + (c^p+d^p)^k \le (a+b)^{pk}+(c+d)^{pk}$. Using this observation with Minkowski's (norm) inequality finishes the proof.
\end{proof}

\begin{theorem}
  \label{thm:hknew}
  Let $\vx, \vy \in \reals_+^n$, $k\ge 1$, and $p \ge 1$. Then,
  \begin{align*}
    [h_k((\vx+\vy)^p)]^{1/{kp}} \le [h_k(\vx^p)]^{1/{kp}} + [h_k(\vy^p)]^{1/{kp}}.
  \end{align*}
\end{theorem}
\begin{proof}
  Given the representation~\eqref{eq:9}, the proof is straightforward. Indeed,
\begin{align*}
  k!h_k((\vx+\vy)^p)
  &= \E[(\gamma_1(x_1+y_1)^p+\cdots+\gamma_n(x_n+y_n)p)^k]\\
  k!^{1/pk}[h_k((\vx+\vy)^p)]^{1/{pk}}
  &=
  \left(\E[(\gamma_1(x_1+y_1)^p+\cdots+\gamma_n(x_n+y_n)^p)^k]\right)^{1/{pk}}\\
  &\le \left(\E[(\bm{\gamma}^T\vx)^{pk}]\right)^{1/{pk}} + \left(\E[(\bm{\gamma}^T\vy)^{pk}]\right)^{1/{pk}}\\
  &= k!^{1/pk}h_k(\vx^{pk})^{1/{pk}} + k!^{1/pk}h_k(\vy^{pk})^{1/{pk}},
\end{align*}
where the inequality follows from Lemma~\ref{lem:mink}.
\end{proof}

A slightly more careful argument with representation~\eqref{eq:9} also yields a ``power'' generalization to the subadditivity result of~\citet{baston1978}  for the ratio $h_k(\vx)/h_1(\vx)$.
\begin{theorem}
  Let $\vx, \vx \in \reals_+^n$, $p\ge 1$, and $k\ge 1$ be an integer. Then,
  \begin{equation}
    \label{eq:11}
    \left[\frac{h_k((\vx+\vy)^p)}{h_{1}((\vx+\vy)^p)}\right]^{\nicefrac1{p(k-1)}}
    \le
    \left[\frac{h_k(\vx^p)}{h_{1}(\vx^p)}\right]^{\nicefrac1{p(k-1)}}
    +
    \left[\frac{h_k(\vx^p)}{h_{1}(\vy^p)}\right]^{\nicefrac1{p(k-1)}}.
  \end{equation}
\end{theorem}
\begin{proof}
  Once again, we rely on the representation~\eqref{eq:9}; this yields
  \begin{equation}
    \label{eq:5}
    \left[\frac{h_k( (\vx+\vy)^p)}{h_{1}((\vx+\vy)^p)}\right]^{\nicefrac1{p(k-1)}}
    =
    \left[\frac{\E[ (\bm{\gamma}^T(\vx+\vy)^p)^k]}{\E[\bm{\gamma}^T(\vx+\vy)^p]}\right]^{\nicefrac1{p(k-1)}}.
  \end{equation}
  To bound~\eqref{eq:5} by the right hand side of~\eqref{eq:11}, it suffices to prove the following inequality:
  \begin{equation}
    \label{eq:13}
    \biggl[\frac{(a^p+b^p)^k+(c^p+d^p)^k}{(a+b)^p+(c+d)^p}\biggr]^{\nicefrac1{p(k-1)}}
    \le
    \biggl[\frac{a^{pk}+c^{pk}}{a^p+c^p}\biggr]^{\nicefrac1{p(k-1)}}
    + 
    \biggl[\frac{b^{pk}+d^{pk}}{b^p+d^p}\biggr]^{\nicefrac1{p(k-1)}}.
  \end{equation}
  To prove~\eqref{eq:13} we use Dresher's generalization of Minkowski's  inequality~\citep{dresher1953,danskin}, which implies that
  \begin{equation}
    \label{eq:8}
    \biggl[\frac{(a^p+b^p)^k+(c^p+d^p)^k}{a^p+b^p+c^p+d^p}\biggr]^{\nicefrac1{k-1}}
    \le
    \biggl[\frac{a^{pk}+c^{pk}}{a^p+c^p}\biggr]^{\nicefrac1{k-1}}
    + 
    \biggl[\frac{b^{pk}+d^{pk}}{b^p+d^p}\biggr]^{\nicefrac1{k-1}}.
  \end{equation}
  Now take $p$-th root on both sides of~\eqref{eq:8} and noting that $p\ge 1$ we obtain
  \begin{align}
    \label{eq:12}
    \biggl[\frac{(a^p+b^p)^k+(c^p+d^p)^k}{a^p+b^p+c^p+d^p}\biggr]^{\nicefrac1{p(k-1)}}
    \le
    \biggl[\frac{a^{pk}+c^{pk}}{a^p+c^p}\biggr]^{\nicefrac1{p(k-1)}}
    + 
    \biggl[\frac{b^{pk}+d^{pk}}{b^p+d^p}\biggr]^{\nicefrac1{p(k-1)}}.
  \end{align}
  Since $p\ge1$, we have $[(a+b)^p+(c+d)^p] \ge a^p+b^p+c^p+d^p$, the left hand side of~\eqref{eq:13} is smaller the left hand side of~\eqref{eq:12}; this concludes the proof.
\end{proof}

\bibliographystyle{abbrvnat}

\begin{thebibliography}{11}
\providecommand{\natexlab}[1]{#1}
\providecommand{\url}[1]{\texttt{#1}}
\expandafter\ifx\csname urlstyle\endcsname\relax
  \providecommand{\doi}[1]{doi: #1}\else
  \providecommand{\doi}{doi: \begingroup \urlstyle{rm}\Url}\fi

\bibitem[Anderson et~al.(1984)Anderson, Morley, and Trapp]{anderson1984}
W.~N. Anderson, T.~D. Morley, and G.~E. Trapp.
\newblock Symmetric function means of positive operators.
\newblock \emph{Linear algebra and its applications}, 60:\penalty0 129--143,
  1984.

\bibitem[Barvinok(2005)]{barvinok2005}
A.~Barvinok.
\newblock Low rank approximations of symmetric polynomials and asymptotic
  counting of contingency tables.
\newblock \emph{arXiv preprint math/0503170}, 2005.

\bibitem[Baston(1978)]{baston1978}
V.~Baston.
\newblock Two inequalities for the complete symmetric functions.
\newblock \emph{Mathematical Proceedings of the Cambridge Philosophical
  Society}, 84\penalty0 (1):\penalty0 1--3, 1978.

\bibitem[Bullen and Marcus(1961)]{bullen1961}
P.~Bullen and M.~Marcus.
\newblock Symmetric means and matrix inequalities.
\newblock \emph{Proceedings of the American Mathematical Society}, 12\penalty0
  (2):\penalty0 285--290, 1961.

\bibitem[Bullen(2013)]{bullen2013}
P.~S. Bullen.
\newblock \emph{Handbook of means and their inequalities}, volume 560.
\newblock Springer Science \& Business Media, 2013.

\bibitem[Danskin(1952)]{danskin}
J.~M. Danskin.
\newblock Dresher's inequality.
\newblock \emph{American Mathematical Monthly}, 59\penalty0 (10):\penalty0
  687--688, 1952.

\bibitem[Dresher et~al.(1953)]{dresher1953}
M.~Dresher et~al.
\newblock Moment spaces and inequalities.
\newblock \emph{Duke mathematical journal}, 20\penalty0 (2):\penalty0 261--271,
  1953.

\bibitem[Marcus and Lopes(1957)]{marcus1957}
M.~Marcus and L.~Lopes.
\newblock Inequalities for symmetric functions and hermitian matrices.
\newblock \emph{Canadian Journal of Mathematics}, 9:\penalty0 305--312, 1957.

\bibitem[Mariet and Sra(2017)]{mariet2017}
Z.~E. Mariet and S.~Sra.
\newblock Elementary symmetric polynomials for optimal experimental design.
\newblock In \emph{Advances in Neural Information Processing Systems}, pages
  2136--2145, 2017.

\bibitem[McLeod(1959)]{mcleod59}
J.~McLeod.
\newblock On four inequalities in symmetric functions.
\newblock \emph{Proceedings of the Edinburgh Mathematical Society}, 11\penalty0
  (4):\penalty0 211--219, 1959.

\bibitem[Muir(1974)]{muir74}
W.~W. Muir.
\newblock Inequalities concerning the inverses of positive definite matrices.
\newblock \emph{Proceedings of the Edinburgh Mathematical Society (Series 2)},
  19:\penalty0 109--113, 1974.

\end{thebibliography}

\end{document}